\documentclass[12pt,a4paper]{amsproc}
\usepackage{amssymb}
\usepackage{setspace} %%final version
\usepackage[pdftex]{graphicx}
\usepackage{caption}
\usepackage{ mathrsfs }
\usepackage{enumerate}

\def\DJ{\leavevmode\setbox0=\hbox{D}\kern0pt\rlap
 {\kern.04em\raise.188\ht0\hbox{-}}D}
\def\dj{\leavevmode
 \setbox0=\hbox{d}\kern0pt\rlap{\kern.215em\raise.46\ht0\hbox{-}}d}

\theoremstyle{plain}
 \newtheorem{thm}{Theorem}
 
 \newtheorem{lem}{Lemma}[section]
 \newtheorem{cor}{Corollary}[section]
\theoremstyle{definition}
 \newtheorem{exm}{Example}[section]
 \newtheorem{rem}{Remark}[section]

 \newtheorem{fgr}{Figure}
\numberwithin{equation}{section}
\numberwithin{equation}{section}

\title[On linear combinations of Chebyshev polynomials]{ON LINEAR COMBINATIONS OF CHEBYSHEV POLYNOMIALS}

\subjclass[2010]{Primary 11B83; Secondary 11R09, 12D10.}

\keywords{Chebyshev polynomials, envelope, Pisot numbers, Salem numbers}

\author[D. Stankov]{\bfseries Dragan Stankov}

\address{
Katedra Matematike RGF-a  \\ %\hfill (Received 00 00 2000)\\
Universitet u Beogradu   \\ %\hfill (Revised  00 00 2000)\\
11000 Beog-\break rad, \DJ u\v sina 7\\
Serbia}
\email{dstankov@rgf.bg.ac.rs}

\begin{document}

%\vspace{18mm}
\setcounter{page}{1}

\begin{abstract}
\normalsize %%final version

We investigate an infinite sequence of polynomials of the form:
\[a_0T_{n}(x)+a_{1}T_{n-1}(x)+\cdots+a_{m}T_{n-m}(x)\]
where $(a_0,a_1,\ldots,a_m)$ is a fixed m-tuple of real numbers, $a_0,a_m\ne0$, $T_i(x)$ are Chebyshev polynomials of the first kind, $n=m,m+1,m+2,\ldots$
Here we analyse the structure of the set of zeros of such polynomial, depending on $A$ and its limit points when $n$ tends to infinity.
Also the expression of envelope of the polynomial is given.
An application in number theory, more precise, in the theory of Pisot and Salem numbers is presented.

\end{abstract}

\maketitle

\section{Introduction}

It is well known \cite{Riv} that the Chebyshev polynomial $T_n(x)$ of the first kind is a polynomial
in $x$ of degree n, defined by the relation

$T_n(x) = \cos n\theta\;\;$ when $x = \cos \theta.$

Let $A=(a_0,a_1,\ldots,a_m)$ be a (m+1)-tuple of real numbers, $a_0, a_m\ne 0$, $m\ge 1$ . We introduce an infinite sequence of polynomials
\[T_{n,A}(x)=a_0T_{n}(x)+a_{1}T_{n-1}(x)+\cdots+a_{m}T_{n-m}(x)\;\;(n\ge m).\]
We will refer to $T_{n,A}(x)$ as an A-Chebyshev polynomial.
We can naturally extend this definition in the case $m=0$ and $A=a_0\ne 0$:
\[T_{n,a0}(x)=a_0T_{n}(x).\]
Also, it will be useful to introduce the polynomial
\[P_{A}(x)=a_0x^m+a_{1}x^{m-1}+\cdots+a_{m}.\]
We will refer to $P_{A}(t)$ as the characteristic polynomial of the A-Chebyshev polynomial.

\section{Roots of A-Chebyshev polynomial}

Let $Z_{n,A}$ denote the set of zeros of the A-Chebyshev polynomial.
The aim of this paper is to analyse the structure of the set $\lim \inf Z_{n,A}$, when $n$ tends to infinity, depending on $A$.
%%$\mathscr{Z}_A=\bigcup_{n=0}^{\infty}Z_{n,A}$, especially the set of its accumulation points $\mathscr{Z}_A'$,
$\lim \inf Z_{n,A}$ consists of those elements which are limits of points in $Z_{n,A}$ for all n. That is, $x \in \lim \inf Z_{n,A}$ if and only if there exists a sequence of points $\{x_k\}$ such that $x_k \in Z_{k,A}$ and $x_k \rightarrow x$ as $k \rightarrow \infty$.

\begin{exm}\label{sec:exm101}
Using our notation, the Chebyshev polynomial $T_n(x)$ is A-Chebyshev polynomial with $A=1$. It is well known \cite{MasHan} that the zeros of $T_n(x)$ are
$x_{n,k} = \cos \frac{(n − k + \frac{1}{2})\pi}{n}$ $(k=1,2,\ldots,n)$. It is obviously that $x_{n,0}$ approaches to -1 and $x_{n,n}$ approaches to 1 when $n\rightarrow\infty$. Since $x_{n,k}$ are equispaced, it is clear that
$\lim \inf Z_{n,1}=[-1,1]$
\end{exm}

\begin{exm}\label{sec:exm102}
What can we say for A-Chebyshev polynomial if $A=(2,-5,2)$?
It is well known \cite{MasHan} that the recurrence relation $T_n(x)=2xT_{n-1}(x)-T_{n-2}(x),\;n=2,3,\ldots$ is satisfied. So $T_{n,(2,-5,2)}(x)=2T_n(x)-5T_{n-1}(x)+2T_{n-2}(x)=4xT_{n-1}(x)-2T_{n-2}(x)-5T_{n-1}(x)+2T_{n-2}(x)$,
$T_{n,(2,-5,2)}(x)=(4x-5)T_{n-1}(x)$. Now it is obvious that $x=\frac{5}{4}$ is a zero of $T_{n,(2,-5,2)}(x)$, for all $n=2,3,\ldots$. So $x=\frac{5}{4}\in Z_{n,(2,-5,2)}$ for all $n=2,3,\ldots$. Taking into account the previous example we conclude that $\lim \inf Z_{n,(2,-5,2)}=[-1,1]\cup\{\frac{5}{4}\}$.
\end{exm}

\begin{lem}\label{sec:ACheb}
$T_{n,A}(x) = \frac{1}{2}(P_A(w)w^{n-m}+P_A(w^{-1})w^{-(n-m)})$ where $w=x+\sqrt{x^2-1}$.

\end{lem}
\begin{proof}
Starting from the definition of A-Chebyshev polynomial and using well known \cite{MasHan} formula $T_n(x) = \frac{1}{2}(w^n+w^{-n})$ we have:
\begin{eqnarray*}
T_{n,A}(x) &= &\sum_{i=0}^m a_iT_{n-i}(x)\\
     &=&\sum_{i=0}^m a_i\frac{1}{2}(w^{n-i}+w^{-n+i})\\
     &=&\frac{1}{2}(\sum_{i=0}^m a_iw^{n-i}+\sum_{i=0}^m a_iw^{-n+i})\\
     &=&\frac{1}{2}(w^{n-m}\sum_{i=0}^m a_iw^{m-i}+w^{-n+m}\sum_{i=0}^m a_iw^{-m+i})\\
     &=& \frac{1}{2}(w^{n-m}P_A(w)+w^{-n+m}P_A(w^{-1})).
\end{eqnarray*}
\end{proof}
One can calculate that if $w=x+\sqrt{x^2-1}$ then $x=\frac{1}{2}(w+w^{-1})$. So, from the previous lemma we can deduce next
\begin{cor}\label{sec:recipZero}
If there iz $w$ such as $ P_A(w)=P_A(w^{-1})=0$ then $T_{n,A}(x)=0$ for $x=\frac{1}{2}(w+w^{-1})$ and for all $n\ge m$.
\end{cor}
In the previous example we can see that $2,\frac{1}{2}$ are roots of $P_A(x)=2x^2-5x+2$, therefore $x=\frac{1}{2}(2+\frac{1}{2})=\frac{5}{4}$ is a zero of $T_{n,A}(x)$ for all $n\ge 2$.

For the next corollary we need the following definition: the set T of Salem numbers is the set of real algebraic
integers $\tau$ greater than 1, such that all its conjugate roots have modulus at most equal to
1, one at least having a modulus equal to 1.
\begin{cor}\label{sec:recipZero1}
If $\tau$ is a Salem number and $ P_A(x)$ is its minimal polynomial then  $T_{n,A}(x)=0$ for $x=\frac{1}{2}(\tau+\tau^{-1})$ and for all $n\ge m$.
\end{cor}
The claim is a direct consequence of a well known property of a Salem number \cite{Ber} that $ P_A(\tau)=P_A(\tau^{-1})=0$.
\begin{thm}\label{cha:nulaPreko1}
If there iz a root $\omega$, out of the unit circle, of the polynomial $P_A$, that is $ P_A(\omega)=0, |\omega|>1$, then for every real number $\varepsilon > 0$, there exists a natural number $n_0$ such that for all $n > n_0$, there is a root $\xi$ of the A-Chebyshev polynomial $T_{n,A}(x)$ such that $|\xi-\frac{1}{2}(\omega+\omega^{-1})|<\varepsilon$.
\end{thm}
\begin{proof}
It is convenient to use Lemma \ref{sec:ACheb} to express  $T_{n,A}(x)$ $ = \frac{1}{2}P_A(w)w^{n-m}+\frac{1}{2}P_A(w^{-1})w^{-(n-m)}$ where $w=x+\sqrt{x^2-1}$, or equivalently  $x=x(w)=\frac{1}{2}(w+w^{-1})$. Since $x(w)$ is continuous for $w>0$, there is $\delta_1 >0$ such that if $|w-\omega|<\delta_1$ then $|\frac{1}{2}(w+w^{-1})-\frac{1}{2}(\omega+\omega^{-1})|<\varepsilon$. We can take an $\delta_2 <|\omega|-1$ such that, in the circle $\{z:|z-\omega|\le \delta_2\}$, there is no root of $P_A(w)$ which is different from $ \omega$.
Let $\delta=\min(\delta_1, \delta_2)$ and $C=\{z:|z-\omega|\le \delta\}$.
Since $\partial C$, the boundary of $C$, is a compact set, $|P_A(w)|$, $|P_A(w^{-1})|$ are continuous on $\partial C$, there is $w_{min}$ where $|P_A(w)|$ attains its minimum, and $w_{max}$ where $|P_A(w^{-1})|$ attains its maximum on $\partial C$. Since $\frac{1}{2}|P_A(w_{max}^{-1})|$ is constant and $|\omega|-\delta >1$, there is $n_0$ such that $\frac{1}{2}|P_A(w_{min})|(|\omega|-\delta)^{n_0-m}>\frac{1}{2}|P_A(w_{max}^{-1})|$.
For $n\ge n_0$, let us denote $f(w)=\frac{1}{2}P_A(w)w^{n-m}$, $g(w)=\frac{1}{2}P_A(w^{-1})w^{-(n-m)}$. This notation corresponds to Rouch\'{e}'s theorem which we intend to use. We have to prove that $|f(w)|>|g(w)|$ on $\partial C$.
Since $|w|\ge |\omega|-\delta >1$ we have on $\partial C$:
\begin{eqnarray*}
|f(w)| &= &\frac{1}{2}|P_A(w)||w|^{n-m}\\
     &\ge&\frac{1}{2}|P_A(w_{min})|(|\omega|-\delta)^{n_0-m}\\
     &>&\frac{1}{2}|P_A(w_{max}^{-1})|\\
     &\ge&\frac{1}{2}|P_A(w^{-1})||w|^{-(n-m)}|\\
     &=& |g(w)|.
\end{eqnarray*}
The conditions in Rouch\'{e}'s theorem are thus satisfied. Consequently, since $f(w)$ has root $\omega$, we conclude that $f(w)+g(w)$ has a root, let it be $\omega_1$, inside the circle $C$ . Clearly, since $|\omega_1-\omega|<\delta_1$, if we denote $\xi=\frac{1}{2}(\omega_1+\omega_1^{-1})$, we conclude
$|\xi-\frac{1}{2}(\omega+\omega^{-1})|<\varepsilon$. Finally, we conclude that $T_{n,A}(\xi)=\frac{1}{2}P_A(\omega_1)\omega_1^{n-m}+\frac{1}{2}P_A(\omega_1^{-1})\omega_1^{-(n-m)}=f(\omega_1)+g(\omega_1)=0$.
\end{proof}
\begin{thm}\label{cha:nulaIspod1}
If $x\in[-1,1]$, then for every real number $\varepsilon > 0$, there exists a natural number $n_0$ such that for all $n > n_0$, there is a root $\xi$ of the A-Chebyshev polynomial $T_{n,A}(x)$ such that $|x-\xi|<\varepsilon$.
\end{thm}
\begin{proof}
Directly from the definitions of the Chebyshev polynomial and the A-Chebyshev polynomial we can show that
\begin{equation}\label{Cheb}
T_{n,A}(x)=a_0\cos n\theta+a_{1}\cos (n-1)\theta+\cdots+a_{m}\cos (n-m)\theta,\;\;n\ge m,
\end{equation} when $x = \cos \theta$.
Since $a_{k}\cos (n-k)\theta=a_{k}\cos (n-m+m-k)\theta=a_{k}(\cos (n-m)\theta\cos (m-k)\theta-\sin (n-m)\theta\sin (m-k)\theta)$ the equation $T_{n,A}(x)=0$ is equivalent with
\[\cos (n-m)\theta\sum_{k=0}^m a_{k}\cos (m-k)\theta=\sin (n-m)\theta\sum_{k=0}^m a_{k}\sin (m-k)\theta.\]
Finally we get \[\tan (n-m)\theta=\frac{\sum_{k=0}^m a_{k}\cos (m-k)\theta}{\sum_{k=0}^m a_{k}\sin (m-k)\theta} .\] The function on the right, let call it $R(\theta)$, does not depend on $n$. The graph of $\tan (n-m)\theta$ consists of parallel equispaced
tangents branches. So if we take $n:=2n-m$ we double $n-m$ and get a new graph which is actually the union of the old one with branches settled in the middle of each pair of neighbouring branches of the old graph. We conclude that all roots of $\tan (n-m)\theta=R(\theta)$, remain to be the roots of $\tan 2(n-m)\theta=R(\theta)$, and new roots interlace with old. Finally, changing variables $\theta=\arccos x$ will preserve order and denseness of the roots.
\end{proof}

\section{Envelope of an A-Chebyshev polynomial}

%We can notice that for any A-Chebyshev polynomial meny roots are in $[-1,1]$, so one can ask why.
Let us observe the Chebishev polynomial $T_n(x)$ again. It is well known that, for any $n$, the graph of
the polynomial oscillates between -1 and 1 when $x\in[-1,1]$. As $n$ increases we have more and more oscillations.
Something like that we have in the case an A-Chebyshev polynomial.
\begin{exm}\label{sec:exm301}
Let $A=(1,0,1)$, so $T_{n,A}(x)=T_{n}(x)+T_{n-2}(x)=2xT_{n-1}(x)-T_{n-2}(x)+T_{n-2}(x)=2xT_{n-1}(x)$. Now it is obvious that $T_{n,A}(x)$ oscillates between lines $y=\pm 2x$, for $x\in [-1,1]$. We will refer to these lines as an envelope of the A-Chebyshev polynomial.
\end{exm}
Using the expression  \eqref{Cheb} we can study the following
\begin{exm}\label{sec:exm302}
Let $A=(1,0,-1)$, so $T_{n,A}(x)=\cos n\theta-\cos ((n-2)\theta)=-2\sin((n-1)\theta) \sin \theta=-2\sin((n-1)\theta)\sqrt{1-\cos^2\theta}=-2\sin((n-1)\theta)\sqrt{1-x^2}$.
Now it is obvious that $T_{n,A}(x)$ oscillates between upper and lower half of the ellipse $y=\pm 2\sqrt{1-x^2}$, for $x\in [-1,1]$. These halves constitute the envelope of the A-Chebyshev polynomial in this case.
\end{exm}
With the same technique we can find the envelope in the next
\begin{exm}\label{sec:exm303}
Let $A=(1,-1)$, so $T_{n,A}(x)=\cos n\theta-\cos ((n-1)\theta)=-2\sin((n-\frac{1}{2})\theta) \sin \frac{\theta}{2}=-2\sin((n-\frac{1}{2})\theta)\sqrt{\frac{1-\cos\theta}{2}}=-\sqrt{2}\sin((n-\frac{1}{2})\theta)\sqrt{1-x}$.
Now it is obvious that the envelope of the A-Chebyshev polynomial is a parabola $y=\pm \sqrt{2}\sqrt{1-x}$, for $x\in [-1,1]$. \end{exm}
Using previous examples, we can formulate the characteristics that an envelope of the A-Chebyshev polynomial must have.

\begin{enumerate}[({Env}1)]
\item The Envelope depends only on $A$. If $A$ is fixed, it is unique for $T_{n,A}(x)$, $n\in \mathbb{N}$.
\item The Envelope is a non negative function.
\item The A-Chebyshev polynomial is not greater in modulus than the envelope, $x\in [-1,1]$.
\item The Envelope is a smooth function except at its zeros.
\item If the envelope and the A-Chebyshev polynomial have equal positive value in $x$ they have also equal the first derivative in $x$.
\end{enumerate}
We shall define the envelope of the A-Chebyshev polynomial as a function which satisfies the characteristics (Env1)-(Env5).
It is naturally to ask how can we find the envelope for an A-Chebyshev polynomial. The next lemma will be useful.
\begin{lem}\label{sec:Koren}
Let $R(t)$, $I(t)$ be real differentiable functions of real argument, with $R'(t)$, $I'(t)$ continuous, $E(t)=\sqrt{R^2(t)+I^2(t)}$, $t\in \mathbb{R}$. Then following three statements are satisfied:
\begin{enumerate}[({}i)]
\item $|R(t)|\le E(t)$,
\item $|R(t)|= E(t)$ if and only if $I(t)=0$,
\item if $I(t)=0$ and $R(t)>0$ then $R(t)= E(t)$ and $R'(t)=E'(t)$.
\end{enumerate}

\end{lem}
\begin{proof}
The first and the second statements are straightforward. To demonstrate that $R'(t)=E'(t)$ we need to determinate
$E'(t)=\frac{2R(t)R'(t)+2I(t)I'(t)}{2\sqrt{R^2(t)+I^2(t)}}$. Using $I(t)=0$, $R(t)>0$ we get the claim.
\end{proof}
\begin{thm}\label{cha:obvojnica}
The envelope $E_A(x)$ for an A-Chebyshev polynomial
$T_{n,A}(x)$ is the square root of the modulus of
\[\sum_{i=0}^m a_i^2+2\sum_{i=0}^{m-1}a_ia_{i+1}T_1(x)+2\sum_{i=0}^{m-2}a_ia_{i+2}T_2(x)+\cdots\]
\[\cdots+2\sum_{i=0}^{m-k}a_ia_{i+k}T_k(x)+\cdots+2a_0a_mT_m(x),\]
or in more compact form \[E_A(x)=\sqrt{\left|\sum_{i=0}^m \sum_{k=0}^m a_ia_kT_{|i-k|}(x)\right|}.\]
\end{thm}
\begin{proof}

Let $z_A(t)=a_0\cos(nt)+a_{1}\cos((n-1)t)+\cdots+a_{m}\cos((n-m)t)+i(a_0\sin(nt)+a_{1}\sin((n-1)t)+\cdots+a_{m}\sin((n-m)t))$
be an auxiliary function on $t\in \mathbb{R}$.
We can see that $T_{n,A}(x)=Re(z_A(t))$ so $|T_{n,A}(x)|^2\le|z_A(t)|^2$, $x=\cos(t)$. We will show that $E_A(x)=|z(t)|$.

%$z_A(t)=a_0e^{int}+a_{1}e^{i(n-1)t}+\cdots+a_{m}e^{i(n-m)t}=e^{i(n-m)t}P_{m,A}(e^{it})$
%$|z_A(t)|^2=z_A(t)\overline{z_A(t)}=e^{i(n-m)t}P_{m,A}(e^{it})e^{-i(n-m)t}P_{m,A}(e^{-it})=P_{m,A}(e^{it})P_{m,A}(e^{-it})=$

\[|z_A(t)|^2=\left (\sum_{k=0}^m a_k\cos((n-k)t)\right )^2+\left (\sum_{k=0}^m a_k\sin((n-k)t)\right )^2\]
\[=\sum_{i=0}^m\sum_{k=0}^m a_i a_k\cos((n-i)t)\cos((n-k)t)+\sum_{i=0}^m\sum_{k=0}^m a_i a_k\sin((n-i)t)\sin((n-k)t)\]
\[=\sum_{i=0}^m\sum_{k=0}^m a_i a_k(\cos((n-i)t)\cos((n-k)t)+\sin((n-i)t)\sin((n-k)t))\]
\[=\sum_{i=0}^m\sum_{k=0}^m a_i a_k\cos((i-k)t).\]

If we substitute $x=\cos(t)$ in $\cos((i-k)t)$ we get $T_{|i-k|}(x)$. Since $E_A(x)$ does not depend on $n$ (Env1) is fulfilled. (Env2), (Env3), (Env4) are straightforward. Using previous lemma if $R(t)=Re(z(t))$, $I(t)=Im(z(t))$,
we can easily obtain (Env5).
\end{proof}
It is useful to calculate the envelope of the A-Chebyshev polynomial for $m=1,2,3,4$. Actually, we give the calculation of the square of the envelope, to avoid cumbersome square roots. Using previous formula we have:
\begin{enumerate}[({m}=1)]

\item $a_0^2+a_1^2+2a_0a_1x$;% for $m=1$;

\item $a_0^2+a_1^2+a_2^2+2(a_0a_1+a_1a_2)x+2a_0a_2(2x^2-1)=$

$=a_0^2+a_1^2+a_2^2-2a_0a_2+(2a_0a_1+2a_1a_2)x+4a_0a_2x^2$;% for $m=2$;

\item $a_0^2+a_1^2+a_2^2+a_3^2+2(a_0a_1+a_1a_2+a_2a_3)x+2(a_0a_2+a_1a_3)(2x^2-1)+2a_0a_3(4x^3-3x)=$

 $=a_0^2+a_1^2+a_2^2+a_3^2-2a_0a_2-2a_1a_3+(2a_0a_1+2a_1a_2+2a_2a_3-6a_0a_3)x+(4a_0a_2+4a_1a_3)x^2+8a_0a_3x^3$;

\item $a_0^2+a_1^2+a_2^2+a_3^2+a_4^2+2(a_0a_1+a_1a_2+a_2a_3+a_3a_4)x+2(a_0a_2+a_1a_3+a_2a_4)(2x^2-1)+2(a_0a_3+a_1a_4)(4x^3-3x)+2a_0a_4(8x^4-8x^2+1)=$

 $=a_0^2+a_1^2+a_2^2+a_3^2+a_4^2-2a_0a_2-2a_1a_3-2a_2a_4+(2a_0a_1+2a_1a_2+2a_2a_3+a_3a_4-6a_0a_3-6a_1a_4)x+(4a_0a_2+4a_1a_3+4a_2a_4-16a_0a_4)x^2+(8a_0a_3+8a_1a_4)x^3+16a_0a_4x^4$.
% for $m=3$
\end{enumerate}
\begin{rem}
There is a connection with the theory of signal processing. The analytic signal $z(t)$ can be expressed in terms of complex polar coordinates, $z(t) = f(t) + i \hat{f}(t) = A(t)e^{i\phi(t)}$ where $A(t) =\sqrt{f^2(t) + \hat{f}^2(t)}$,
and $\phi(t) = \arctan \frac{\hat{f}(t)}{f(t)}$. These functions are respectively called the amplitude envelope and instantaneous phase of the signal, $\hat{f}(t)$ is Hilbert transform of $f(t)$.
\end{rem}
\section{Connection between the envelope and the characteristic polynomial}

Until now we used the envelope to describe the graph of the A-Chebyshev polynomial if $x$ is of modulus not greater than 1. If $|x|>1$ we preferred the characteristic polynomial $P_A(x)$. It is natural to ask, is there any connection between the envelope and the characteristic polynomial of the A-Chebyshev polynomial. The next theorem shows that the answer is affirmative.

\begin{thm}\label{cha:odnos}
The envelope of the A-Chebyshev polynomial is the function \[E_A(x)=\sqrt{\left|P_A(x+\sqrt{x^2-1})P_A(x-\sqrt{x^2-1})\right|}.\]
\end{thm}
\begin{proof}
We shall start from the compact form of the envelope given in Theorem \ref{cha:obvojnica} and use well known \cite{MasHan} formula $T_n(x) = \frac{1}{2}(w^n+w^{-n})$ where $w=x+\sqrt{x^2-1}$.
\begin{eqnarray*}
E_A(x) &=&\sqrt{\left|\sum_{i=0}^m \sum_{k=0}^m a_ia_kT_{|i-k|}(x)\right|}\\
     &=&\sqrt{\left|\sum_{i=0}^m \sum_{k=0}^m a_ia_k\frac{1}{2}(w^{i-k}+w^{-(i-k)})\right|}\\
     &=&\sqrt{\left|\sum_{K=0}^m \sum_{I=0}^m \frac{1}{2}a_Ka_Iw^{K-I}+\sum_{i=0}^m \sum_{k=0}^m \frac{1}{2}a_ia_kw^{k-i}\right|}.%\;\;(K:=i,I:=k).\\
\end{eqnarray*}
(Here we renamed $i$ with K and $k$ with $I$ in the first double sum. Now we shall switch the order of summing in the first double sum and apply obvious $a_Ka_I=a_Ia_K$.)
\begin{eqnarray*}
     &=&\sqrt{\left|\sum_{I=0}^m \sum_{K=0}^m \frac{1}{2}a_Ia_Kw^{K-I}+\sum_{i=0}^m \sum_{k=0}^m \frac{1}{2}a_ia_kw^{k-i}\right|}\\
     &=&\sqrt{\left|2\sum_{i=0}^m \sum_{k=0}^m \frac{1}{2}a_ia_kw^{k-i}\right|}\\
     &=&\sqrt{\left|\sum_{i=0}^m \sum_{k=0}^m a_ia_kw^{m-i}w^{-m+k}\right|}\\
%\end{eqnarray*}
%\begin{eqnarray*}
     &=&\sqrt{\left|\sum_{i=0}^m a_iw^{m-i}\sum_{k=0}^m a_kw^{-m+k}\right|}\\
     &=&\sqrt{\left|P_A(w)P_A(w^{-1})\right|}.\\
     &=& \sqrt{\left|P_A(x+\sqrt{x^2-1})P_A(x-\sqrt{x^2-1})\right|}.
\end{eqnarray*}
\end{proof}
Figure 1. shows graphs of A-Chebyshev polynomials of the first kind $T_{14,(1,0,0,1)}(x)$, $T_{44,(1,0,0,1)}(x)$ together with their common envelope $E(x)=\sqrt{|2+6x-8x^3|}$.
\begin{fgr}
\begin{center}
\includegraphics {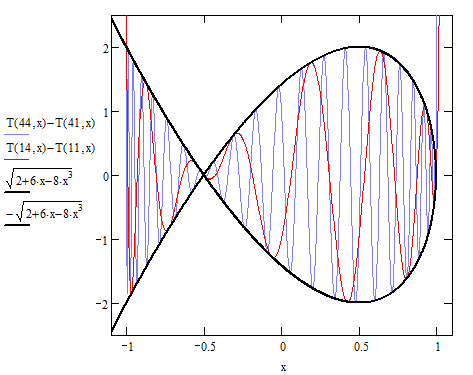}
%\caption{A-Chebyshev of the first kind}
\end{center}
\end{fgr}

\section{A-Chebyshev polynomial of the second kind}

It is well known \cite{Riv} that the Chebyshev polynomial $U_n(x)$ of the second kind is a polynomial
in $x$ of degree n, defined by the relation

$U_n(x) = \frac{\sin (n+1)\theta}{\sin\theta}\;\;$ when $x = \cos \theta.$

Let $A=(a_0,a_1,\ldots,a_m)$ be a (m+1)-tuple of real numbers, $a_0, a_m\ne 0$, $m\ge 1$ . We introduce an infinite sequence of polynomials
\[U_{n,A}(x)=a_0U_{n}(x)+a_{1}U_{n-1}(x)+\cdots+a_{m}U_{n-m}(x)\;\;(n\ge m).\]
We will refer to $U_{n,A}(x)$ as an A-Chebyshev polynomial of the second kind.
We can naturally extend this definition in the case $m=0$ and $A=a_0\ne 0$:
\[U_{n,a0}(x)=a_0U_{n}(x).\]
We will refer to the polynomial
\[P_{A}(x)=a_0x^m+a_{1}x^{m-1}+\cdots+a_{m}\]
as the characteristic polynomial of the A-Chebyshev polynomial.

\begin{lem}\label{sec:AChebII}
$U_{n,A}(x) = \frac{1}{w-w^{-1}}(w^{n+1-m}P_A(w)-w^{-n-1+m}P_A(w^{-1}))$ where $w=x+\sqrt{x^2-1}$.

\end{lem}
\begin{proof}
Starting from the definition of A-Chebyshev polynomial and using well known \cite{MasHan} formula $U_n(x) = \frac{w^{n+1}-w^{-n-1}}{w-w^{-1}}$ we have:
\begin{eqnarray*}
U_{n,A}(x) &= &\sum_{i=0}^m a_iU_{n-i}(x)\\
     &=&\sum_{i=0}^m a_i\frac{w^{n+1-i}-w^{-n-1+i}}{w-w^{-1}}\\
     &=&\frac{1}{w-w^{-1}}(\sum_{i=0}^m a_iw^{n+1-i}-\sum_{i=0}^m a_iw^{-n-1+i})\\
     &=&\frac{1}{w-w^{-1}}(w^{n+1-m}\sum_{i=0}^m a_iw^{m-i}-w^{-n-1+m}\sum_{i=0}^m a_iw^{-m+i})\\
     &=& \frac{1}{w-w^{-1}}(w^{n+1-m}P_A(w)-w^{-n-1+m}P_A(w^{-1})).
\end{eqnarray*}
\end{proof}

\begin{thm}\label{cha:nulaPreko1II}
If there is a root $\omega$, out of the unit circle, of the polynomial $P_A$, that is $ P_A(\omega)=0, |\omega|>1$, then for every real number $\varepsilon > 0$, there exists a natural number $n_0$ such that for all $n > n_0$, there is a root $\xi$ of the A-Chebyshev polynomial of the second kind $U_{n,A}(x)$ such that $|\xi-\frac{1}{2}(\omega+\omega^{-1})|<\varepsilon$.
\end{thm}
\begin{proof}
It is convenient to use the previous lemma to express  $U_{n,A}(x)$ = $ \frac{1}{w-w^{-1}}(w^{n+1-m}P_A(w)-w^{-n-1+m}P_A(w^{-1}))$ where $w=x+\sqrt{x^2-1}$ or equivalently  $x=x(w)=\frac{1}{2}(w+w^{-1})$. Since $x(w)$ is continuous for $w>0$, there is $\delta_1 >0$ such that if $|w-\omega|<\delta_1$ then $|\frac{1}{2}(w+w^{-1})-\frac{1}{2}(\omega+\omega^{-1})|<\varepsilon$. We can take an $\delta_2 <|\omega|-1$ such that, in the circle $\{z:|z-\omega|\le \delta_2\}$, there is no root of $P_A(w)$ which is different from $ \omega$.
Let $\delta=\min(\delta_1, \delta_2)$ and $C=\{z:|z-\omega|\le \delta\}$.
Since $\partial C$, the boundary of $C$, is a compact set, $|P_A(w)|$, $|P_A(w^{-1})|$ are continuous on $\partial C$, there is $w_{min}$ where $|P_A(w)|$ gets its minimum and $w_{max}$ where $|P_A(w^{-1})|$ gets its maximum on $\partial C$. Since $\frac{1}{w-w^{-1}}|P_A(w_{max}^{-1})|$ is constant and $|\omega|-\delta >1$, there is $n_0$ such that $\frac{1}{w-w^{-1}}|P_A(w_{min})|(|\omega|-\delta)^{n_0+1-m}>\frac{1}{w-w^{-1}}|P_A(w_{max}^{-1})|$.
For $n\ge n_0$ let us denote $f(w)=\frac{1}{w-w^{-1}}w^{n+1-m}P_A(w)$, $g(w)=\frac{-1}{w-w^{-1}}w^{-n-1+m}P_A(w^{-1})$. This notation corresponds with Rouch\'{e}'s theorem which we intend to use. We have to prove that $|f(w)|>|g(w)|$ on $\partial C$.
Since $|w|\ge |\omega|-\delta >1$ we have on $\partial C$:
\begin{eqnarray*}
|f(w)| &= &\frac{1}{w-w^{-1}}|P_A(w)||w|^{n+1-m}\\
     &\ge&\frac{1}{w-w^{-1}}|P_A(w_{min})|(|\omega|-\delta)^{n_0+1-m}\\
     &>&\frac{1}{w-w^{-1}}|P_A(w_{max}^{-1})|\\
     &\ge&\frac{1}{w-w^{-1}}|P_A(w^{-1})||w|^{-(n+1-m)}|\\
     &=& |g(w)|.
\end{eqnarray*}
The conditions in Rouch\'{e}'s theorem are thus satisfied. Consequently, since $f(w)$ has root $\omega$, we conclude that $f(w)+g(w)$ has a root, let it be $\omega_1$, inside the circle $C$ . Clearly, since $|\omega_1-\omega|<\delta_1$, if we denote $\xi=\frac{1}{2}(\omega_1+\omega_1^{-1})$, we conclude
$|\xi-\frac{1}{2}(\omega+\omega^{-1})|<\varepsilon$. Finally, we conclude that $U_{n,A}(\xi)=\frac{1}{w-w^{-1}}P_A(\omega_1)\omega_1^{n+1-m}-\frac{1}{w-w^{-1}}P_A(\omega_1^{-1})\omega_1^{-(n+1-m)}=f(\omega_1)+g(\omega_1)=0$.
\end{proof}

\section{An application in number theory}

%Definition 5.2.1. The set S of Pisot numbers is the set of real algebraic
%integers () greater than 1 whose other conjugates have modulus strictly smaller
%than 1.

Recall that $q>1$ is a Pisot number if $q$ is an algebraic
integer, whose other conjugates are of modulus strictly less
than 1.
%Recall also that $\tau>1$ is a Salem number if $\tau$ is an
%algebraic integer, whose conjugates are of modulus at most 1 and with a conjugate
%of modulus 1.
Salem proved that every Pisot number is a limit point of the set $T$ of Salem numbers.
Let $Q(x)=x^mP(\frac{1}{x})=a_0+a_{1}x+\cdots+a_{m}x^m$ be the reciprocal polynomial of the polynomial $P(x)=a_0x^m+a_{1}x^{m-1}+\cdots+a_{m}$. Salem showed that if $P(x)$ is the minimal polynomial of a Pisot number $q$ then $R_k(x)=x^kP(x) + Q(x)$ is polynomial with a root $\tau_k$ that is a Salem number, and the limit of the sequence $\tau_k$ is  $q$, $k\rightarrow\infty$. There is a connection between Salem sequence $R_k(x)$ with A-Chebyshev polynomials $T_{n,A}(x)$ the characteristic polynomial of which is $P(x)$. We have seen that $T_{n,A}(\frac{1}{2}(w^{}+w^{-1}))= T_{n,A}(x)= \sum_{i=0}^m a_iT_{n-i}(x)=\sum_{i=0}^m a_i\frac{1}{2}(w^{n-i}+w^{-n+i})$. Now we can show that $2w^nT_{n,A}(\frac{1}{2}(w^{}+w^{-1}))=R_{2n-m}(w)$. Really, we obtain $2w^nT_{n,A}(\frac{1}{2}(w^{}+w^{-1}))=\sum_{i=0}^m a_i(w^{2n-i}+w^{i})=w^{2n-m}\sum_{i=0}^m a_i(w^{m-i})+\sum_{i=0}^m a_i(w^{i})=w^{2n-m}P(w)+Q(w)$.

The question is what is going on if we use A-Chebyshev polynomial of the second kind instead of $T_{n,A}(x)$.
We will demonstrate that one more sequence of Salem numbers, which converges to the Pisot number $q$, appears.
It is obvious that $U_{n,A}(\frac{1}{2}(w^{}+w^{-1}))= U_{n,A}(x)= \sum_{i=0}^m a_iU_{n-i}(x)=\sum_{i=0}^m \frac{a_i}{w-w^{-1}}(w^{n+1-i}-w^{-n-1+i})=\sum_{i=0}^m a_i(w^{n-i}+w^{n-i-2}+w^{n-i-4}\cdots+w^{-n+i+2}+w^{-n+i}))$. We claim that $S_{2n}(w)=w^nU_{n,A}(\frac{1}{2}(w^{}+w^{-1}))$ is polynomial of degree $2n$ with a root $\tau_{2n}$ that is a Salem number, and the limit of the sequence $\tau_{2n}$ is  $q$, $n\rightarrow\infty$. Using the previous theorem it is clear that there is a root $\tau_{2n}$ of $S_{2n}(w)$ such as $\tau_{2n}\rightarrow q$. It is obvious that $S_{2n}(w)$ is a reciprocal polynomial. It remains to be proved that all other roots of $S_{2n}(w)$ are in the unit circle. We shall apply the method Salem (communicated by Hirschman) used to prove the same property of his sequence $R_{k}(x)$ \cite{Sal}. Using Lemma \ref{sec:AChebII} we have
\begin{eqnarray*}
S_{2n}(w)&=&w^nU_{n,A}(\frac{1}{2}(w^{}+w^{-1}))\\  		  	                   &=&\frac{w^n}{w-w^{-1}}(w^{n+1-m}P_A(w)-w^{-n-1+m}P_A(w^{-1}))\\
     &=&\frac{w^{n+1}}{w^2-1}(w^{n+1-m}P_A(w)-w^{-n-1+m}P_A(w^{-1}))\\
     &=&\frac{1}{w^2-1}(w^{2n+2-m}P_A(w)-w^{m}P_A(w^{-1}))\\
     &=&\frac{1}{w^2-1}(w^{2n+2-m}P_A(w)-Q(w)).
\end{eqnarray*}
We denote by $\epsilon$ a positive number and consider the equation
\[(1 + \epsilon)w^{2n+2-m}P_A(w)- Q(w) = 0.\]
Since for $ |w|=1$ we have $|P(w)|=|Q(w)|$, it follows by Rouch\'{e}'s theorem that inside
the circle $ |w|=1$ the number of roots of the last equation is equal to the number
of roots of $w^{2n+2-m}P(w) $, that is, $(2n+2-m) +m - 1$. As $\epsilon \rightarrow 0$, these roots vary continuously.
Hence, for $\epsilon=0$ we have $2n+1$ roots with modulus $\le 1$. It is obvious that two roots are $1,-1$ so the fraction can be reduced with $w^2-1$. Finally we conclude that at most one root of $S_{2n}(w)$ is outside the unit circle.
\begin{exm}\label{sec:exmSal}
Let $q$ be the golden ratio, the greater root of $P(x)=x^2-x-1$.
Then $R_k(w)=w^{k+2}-w^{k+1}-w^k-w^2-w+1$ is the Salem's sequence.
Our sequence of Salem numbers $\tau_{2m}$ which converge to $q$ consists of the greatest in modulus roots of the polynomials
$S_{2n}(w)=\frac{1}{w^2-1}(w^{2n+2-m}P_A(w)-Q(w))=\frac{1}{w^2-1}(w^{2n}(w^2-w-1)+w^2+w-1)=$
$(w^{2n}+w^{2n-2}+w^{2n-4}\cdots+w^2+1)-(w^{2n-1}+w^{2n-3}+w^{2n-5}\cdots+w^{3}+w^{})-(w^{2n-2}+w^{2n-4}+w^{2n-6}\cdots+w^{4}+w^{2}).$
Finally \[S_{2n}(w)=w^{2n}-(w^{2n-1}+w^{2n-3}+w^{2n-5}\cdots+w^{3}+w^{})+1.\]
\end{exm}

%\bibliography{BiblSpect}
%\bibliographystyle{abbrv}

\end{document}